\documentclass[12pt, twoside]{article}
\usepackage{times}
\usepackage{enumerate}
\def\titlerunning#1{\gdef\titrun{#1}}
\makeatletter
\def\author#1{\gdef\autrun{\def\and{\unskip, }#1}\gdef\@author{#1}}
\def\address#1{{\def\and{\\\hspace*{18pt}}\renewcommand{\thefootnote}{}%
\footnote {#1}}%
\markboth{\autrun}{\titrun}}
\makeatother
\def\email#1{e-mail: #1}
\def\subjclass#1{{\renewcommand{\thefootnote}{}%
\footnote{\emph{Mathematics Subject Classification (2010):} #1}}}

\usepackage{amsmath}
\usepackage{amsfonts}
\usepackage{amssymb}
\usepackage[english]{babel}
\usepackage{graphicx}
\usepackage{amsthm}
\usepackage{xcolor}

\newtheorem{thm}{Theorem}[section]
\newtheorem{cor}[thm]{Corollary}
\newtheorem{lem}[thm]{Lemma}
\newtheorem{prop}[thm]{Proposition}
\newtheorem{ex}[thm]{Example}
\newtheorem{notation}[thm]{Notation}

\usepackage{wasysym}

\usepackage{bbm}

\theoremstyle{definition}
\newtheorem{defn}[thm]{Definition}

\theoremstyle{remark}
\newtheorem{rem}[thm]{Remark}

\numberwithin{equation}{section}

\newcommand{\F}{\mathbb{F}}

\newcommand{\Z}{\mathbb{Z}}

\begin{document}
\baselineskip=17pt
\titlerunning{Partitioned difference families}
\title{Partitioned difference families\\
and harmonious linear spaces}

\author{Marco Buratti
\and
Dieter Jungnickel}

\date{}

\maketitle

\address{
M. Buratti: Dipartimento di Scienze di Base e Applicate per l’Ingegneria (S.B.A.I.), Sapienza
Universit\`a di Roma, Via Antonio Scarpa, 10, Italy; \email{marco.buratti@uniroma1.it}
\and
Dieter Jungnickel: Mathematical Institute, University of Augsburg, 86135 Augsburg, Germany;
\email{jungnickel@math.uni-augsburg.de}
}

\subjclass{05B05, 05B10}

\begin{abstract}
We say that a linear space is  {\it harmonious} if it is resolvable and admits an automorphism
group acting sharply transitively on the points and transitively on the parallel classes. 
Generalizing old results by the first author et al. we present some difference methods to construct harmonious linear spaces.
We prove, in particular, that for any finite non-singleton subset $K$ of $\Z^+$ there are infinitely many values of
$v$ for which there exists a {\it partitioned difference family} that is the base parallel class of a harmonious linear space with $v$ 
points whose block sizes are precisely the elements of $K$.
\end{abstract}

\noindent {\textbf{Keywords:}}  linear space; automorphism group;
partitioned difference family; strong difference family; cyclotomy.

\section{Introduction}
We recall that a {\it linear space} S$(2,K,v)$ is a pair $(V,{\cal B})$ where $V$ is a set of $v$ ``points" and $\cal B$ is
a set of subsets of $V$ called {\it blocks} with sizes belonging to $K$ and the property that
any 2-subset of $V$ is contained in exactly one block. A S$(2,K,v)$ is {\it mandatory} if the set of block sizes is precisely $K$.
Throughout this paper, speaking of a S$(2,K,v)$ it will be tacitly understood that it is mandatory.
When $K$ is a singleton $\{k\}$ one speaks of a
{\it Steiner 2-design} and writes S$(2,k,v)$ (see, e.g., \cite{BJL}).

A {\it resolution} of a linear space is a partition $\cal R$ of its block set into classes ({\it parallel classes}) 
each of which is, in its turn, a partition of the point set. A linear space is {\it resolvable} if it admits at
least one resolution. Speaking of a {\it resolved} linear space one means a triple $(V,{\cal B},{\cal R})$ 
where $(V,{\cal B})$ is a linear space and $\cal R$ is a specific resolution of it.
An automorphism group of such a resolved linear space is a subgroup of $Sym(V)$
leaving $\cal R$ (and consequently $\cal B$) invariant.

The construction of a resolvable Steiner 2-design is not easy in general. Suffice it to say that for any fixed $k\geq5$ the set of values of $v$ for which
there is a resolvable S$(2,k,v)$ is still unknown. 
Anyway it was proved in \cite{RW} (see also \cite{BJL}, p. 801) 
that a resolvable S$(2,k,v)$ certainly exists if $v$ is admissible -- which means $v\equiv k$ (mod $k(k-1)$) -- and ``sufficiently large".
As far as we are aware the analogous result for general linear spaces is still lacking.

We propose the following definitions.
\begin{defn}\label{def}
A {\it harmonious} linear space S$(2,K,v)$, briefly a HS$(2,K,v)$, is a resolved S$(2,K,v)$ with an automorphism group
acting sharply transitively on the points and transitively on the parallel classes. 

An {\it almost harmonious} linear space S$(2,K,v)$, briefly an AHS$(2,K,v)$, is a resolved S$(2,K,v)$ with an automorphism group
acting sharply transitively on all but one point and transitively on the parallel classes. 
\end{defn}

\begin{ex}
Here is the resolution ${\cal R}=\{{\cal P}_0,{\cal P}_1,{\cal P}_2,{\cal P}_3\}$ of a HS$(2,\{2,3\},8)$ with point set $\Z_8$:
$${\cal P}_0=\bigl{\{}\{0,4\}, \ \{1,6,7\}, \ \{5,2,3\}\bigl{\}};$$
$${\cal P}_1=\bigl{\{}\{1,5\}, \ \{2,7,0\}, \ \{6,3,4\}\bigl{\}};$$
$${\cal P}_2=\bigl{\{}\{2,6\}, \ \{3,0,1\}, \ \{7,4,5\}\bigl{\}};$$
$${\cal P}_3=\bigl{\{}\{3,7\}, \ \{4,1,2\}, \ \{0,5,6\}\bigl{\}}.$$
Indeed it is readily seen that we have ${\cal P}_i+1={\cal P}_{i+1 \ (mod \ 4)}$, hence 
the group $\widehat\Z_8$ of translations modulo 8 acts sharply transitively on $\Z_8$ and
transitively on $\cal R$.
\end{ex}

If in the above example we add an extra point, say $\infty$, to the blocks of size 2 we get an AHS$(2,3,9)$.

It is easy to see that a Steiner 2-design cannot be harmonious. On the other hand, several classes of almost 
harmonious Steiner 2-designs have been constructed in \cite{BYW,BZ}.
In particular, it was proved in \cite{BYW} that for any integer $k$ there are infinitely many values of $v$
for which there exists an almost harmonious S$(2,k,v)$.

In this paper we give some constructions for harmonious linear spaces and, as main result, we prove
that for any finite non-singleton set $K$ of positive integers, there are infinitely many values of $v$ for which there 
exists a HS$(2,K,v)$.

The starter parallel class of each of our harmonious linear spaces is a
{\it partitioned difference family}, that is a partition of a group in blocks 
whose lists of differences cover every non-identity element of the group a constant number of times.
Thus, as a consequence of our main result, for any finite non-singleton 
set $K$ of positive integers  there are infinitely many values of $v$ for which there exists
a partitioned difference family in a group of order $v$ with set of block sizes equal to $K$.

\section{Use of difference families}

The {\it list of differences} of a multisubset $B=\{b_1,\dots,b_k\}$ of an additive group $G$
is the multiset $\Delta B=\{b_i-b_j \ | \ 1\leq i, j \leq k; i\neq j\}$. 
More generally,
if $\cal F$ is a collection of multisubsets of $G$,  its list of differences is the multiset union
of the lists of differences of all its members:
 $\Delta{\cal F}:=\bigcup_{B\in{\cal F}}\Delta B$.
 The multiset union of all the members of ${\cal F}$ will be called the
{\it flatten} of $\cal F$.

If $\Delta{\cal F}$ is $\lambda$ times $G\setminus H$ with $H$
a subgroup of $G$, then ${\cal F}$ is a {\it difference family in $G$ relative to $H$ with index $\lambda$}.
In this case one briefly says that $\cal F$ is a $(G,H,K,\lambda)$-DF where $K$ denotes the multiset of all sizes of the members of $\cal F$.

If $\Delta{\cal F}$ is $\lambda$ times all of $G$ (zero included) then ${\cal F}$ is a $(G,K,\lambda)$ 
{\it strong difference family} (SDF) in $G$ and we briefly say that it is a $(G,K,\lambda)$-SDF with $K$ defined as above. 
Note that the index of a SDF is necessarily even.

The members of a DF or SDF are called {\it blocks}.
It is evident that every block of a difference family of index 1 is a set, not a multiset. In contrast at least one block of a SDF has some repeated elements.
A DF or SDF whose blocks have all the same size $k$ is said to be {\it uniform} and one usually writes 
$(G,H,k,\lambda)$-DF or $(G,k,\lambda)$-SDF, respectively.

Speaking of a $(G,K,\lambda)$-DF (without the adjective ``strong") one usually means a
$(G,H,K,\lambda)$-DF with $H=\{0\}$. A difference family with only one block is called a {\it difference set} \cite{DPS}.
A $(G,K,\lambda)$-DF whose blocks partition $G$ 
(or, equivalently, whose flatten is $G$) is said to be {\it partitioned} (PDF).

Partitioned difference families have been introduced by Ding and Ying \cite{DY} in view of 
their useful applications to {\it constant composition codes}.
The equivalent notion of a {\it zero difference balanced function} (ZDBF) was introduced a few years 
later by Ding \cite{D}. As pointed out in \cite{BJ,BJ2}, this double terminology  unfortunately caused some 
confusion. Indeed several authors presented some seemingly new ZDBFs which,
as a matter of fact, could be seen as PDFs known since a long time (see also \cite{BDCC}).

The following result is very well known. 
\begin{lem}\label{DF->LS}
Let $H$ be a subgroup of order $u$ of a group $G$ of order $v$, and let $\cal F$ be a $(G,H,K',1)$-DF.
Then $$\{B+g \ | \ B\in{\cal F}; g\in G\} \ \cup \ \{\mbox{right cosets of $H$ in $G$}\}$$ 
is the block set of a S$(2,K,v)$ where $K$ is the underlying set of $K' \ \cup \ \{u\}$.
\end{lem}

In \cite{BZ}, a $(G,H,k,1)$-DF has been called {\it resolvable} if its flatten is a complete system 
of representatives for the non-trivial left cosets of $H$ in $G$. 
Here, we extend this notion also to non-uniform DF's.
 \begin{defn}
 A $(G,H,K,1)$-DF is {\it resolvable} if its flatten is a complete system of representatives
for the non-trivial left cosets of $H$ in $G$.
 \end{defn}
 
 Throughout the paper we agree to use the following notation.

\begin{notation}
Multisets will be often denoted in ``exponential notation". Hence, writing $K=\{k_1^{\mu_1},...,k_t^{\mu_t}\}$ we mean
the multiset with underlying set $\{k_1,...,k_t\}$ where each $k_i$ has multiplicity $\mu_i$. The exponent $\mu_i$ will be omitted
when it is equal to $1$. If $n$ is a positive integer and $K$ is a multiset, then  $\underline{n}K$ denotes
the multiset union of $n$ copies of $K$. Thus, given $K$ as above, we have $\underline{n}K=\{k_1^{\mu_1n},...,k_t^{\mu_tn}\}$. 
The multiset consisting of the zero element of an additive group repeated a certain number $n$ of times, that is
$\underline{n}\{0\}=\{0^n\}$, will be denoted by $O_n$.
\end{notation}

 The proof of the ``only if" part of Theorem 1 in \cite{BZ} can be easily adapted to obtain the following.

\begin{thm}\label{RDF->PDF}
Let $H$ be a subgroup of order $u$ of an additive group $G$ of order $v$, 
and let $\cal F$ be a resolvable $(G,H,K',1)$-DF. Then  ${\cal P}=\{B+h \ | \ B\in{\cal F},h\in H\} \ \cup \ \{H\}$
is a $(G,\underline{u}K' \ \cup \ \{u\},u)$-PDF which is the base parallel class of a HS$(2,K,v)$ where $K$
is the underlying set of $K'\cup\{u\}$.
\end{thm}
\begin{proof}
We have $\Delta{\cal F}=G\setminus H$ by definition of a $(G,H,K',1)$-DF.
Of course, for any $B\in {\cal F}$ and any $h\in H$ we have $\Delta(B+h)=\Delta B$.
Hence the list of differences of $\{B+h \ | \ B\in{\cal F},h\in H\}$ is $|H|$ times $G\setminus H$:
$$\Delta\{B+h \ | \ B\in{\cal F},h\in H\}=\underline{u}(G\setminus H).$$
Also, it is evident that $\Delta H=\underline{u}(H\setminus\{0\})$. The last two equalities give
$\Delta{\cal P}=\underline{u}(G\setminus\{0\})$, i.e., ${\cal P}$ is a DF in $G$ of index $u$.

The fact that $\cal F$ is resolvable implies that the flatten of $\{B+h \ | \ B\in{\cal F},h\in H\}$
is precisely $G\setminus H$ and then the flatten of $\cal P$ is precisely $G$. We conclude that ${\cal P}$ is
a $(G,\underline{u}K' \ \cup \ \{u\},u)$-PDF.

Let $G_{\cal P}$ be the stabilizer of $\cal P$ under the natural right action of $G$.
It is obvious that $H\subset G_{\cal P}$.
Let us prove the reverse inclusion. Assume that $g\in G_{\cal P}\setminus H$. In this case
$H+g\in{\cal P}$ (since $H$ is a member of $\cal P$) but $H+g\neq H$. Thus we have $H+g=B+h$ for a suitable pair $(B,h)\in{\cal F}\times H$.
It follows that $\Delta(H+g)=\Delta(B+h)$, hence $\Delta H=\Delta B$. Considering that $\Delta B$ is disjoint with 
$H$ and that every element of $\Delta H$ is in $H$, we infer that $H=\{0\}$, that
is the case where $\Delta H$ is empty. So $\cal P$ is a PDF of index 1.
Let $B=\{b,b+g,\dots\}$ be the block of $\cal F$ such that $g\in\Delta B$. 
We have $B+g\in{\cal P}$ since $g\in G_{\cal P}$ and we see that $b+g\in B\cap (B+g)$ so that
$B+g=B$ since the blocks of $\cal P$ are pairwise disjoint.
Hence $g$ stabilizes $B$ which means that $B$ is a union of right cosets of
$\langle g\rangle$ in $G$. It follows that $\Delta B$ contains $\Delta \langle g\rangle$ which is $o(g)$ times all
the elements of $\langle g\rangle\setminus\{0\}$. This contradicts the fact that $\cal P$ has index 1.
We conclude that $G_{\cal P}=H$. 
Hence, if $S$ is a complete system of representatives for the right cosets of $H$ in $G$, 
the $G$-orbit of $\cal P$ is $\{{\cal P} + s \ | \ s\in S\}$. Now note that $\bigcup_{s\in S}({\cal P}+s)$
is precisely the block set of the S$(2,K,v)$ generated by $\cal F$ (see Lemma \ref{DF->LS}). Also, given 
that $\cal P$ is partitioned, it is clear that ${\cal P} + s$ is a parallel class of this S$(2,K,v)$ for every $s\in S$
and then $\{{\cal P} + s \ | \ s\in S\}$ is a resolution of it.
The assertion follows.
\end{proof}

\section{Harmonious strong difference families}

Uniform strong difference families have been implicitly used in the literature for many years
and they have been formally introduced for the first time in \cite{SDF}. Since then, they have
been explicitly used in many papers \cite{BBGRT,BN,BP,BYW,CCFW,CFW1,CFW2,GFW,Momihara,YYL}. Non-uniform SDFs have been used in \cite{HPDF} 
for the construction of the so-called {\it Hadamard} partitioned difference families also considered in \cite{N}. 
 
The following definition is crucial in this paper.

\begin{defn}
A $(G,K,\lambda)$-SDF is {\it harmonious} if its block sizes sum up to $\lambda$ 
(that is, if its flatten has size $\lambda$).
\end{defn}

A $(G,k,\lambda)$-SDF with $n$ blocks has flatten of size $kn$ and we have $\lambda|G|=k(k-1)n$. We deduce that a uniform $(G,k,\lambda)$-SDF 
is harmonious if and only if $|G|=k-1$. 




\begin{ex}\label{classic}
If $G$ is any additive group of order $k-1$, then $\Sigma=\{O_k,(G\cup\{0\})^k\}$
is a harmonious $(G,k,k^2+k)$-SDF.
\end{ex}

\begin{ex}\label{9}
Let $G=\Z_9$. Then
$$\{0^3,1\}\quad \mbox{and} \quad \{0^3,1,2^2,3,4,5,6^2,7\}$$
are the blocks of a harmonious $(G,\{4,12\},16)$-SDF. Also,
$$O_3 \ \cup \ \{1,4,7\}\quad \mbox{and} \quad O_3 \ \cup \ G$$
are the blocks of a harmonious $(G,\{6,12\},18)$-SDF.
\end{ex}

\begin{ex}\label{4,14}
Let $G=\{0,1,a,b\}$ be one of the two groups of order $4$.
Then $$\{O_3, \ \{0^2,1,a,b\}, \ \{1^2,a^2,b^2\}\}$$ is a
harmonious $(G,\{3,5,6\},14)$-SDF and
$$\{O_3, \ \{1,a,b\}, \ \{0^3,1,a,b\}^2\}$$
is a harmonious $(G,\{3^2,6^2\},18)$-SDF.
\end{ex}

\begin{ex}\label{3,18}
Let $G$ be any group of order $k$.
Then $\Sigma=\{O_k^2, \ G^{2k-4}, \ \underline{2}G\}$ is a harmonious $(G,\{k^{2k-2},2k\},2k^2)$-SDF.
For instance, 
$$\{\{0,0,0\}, \ \{0,0,0\}, \ \{0,1,2\}, \  \{0,1,2\}, \  \{0,0,1,1,2,2\}\}$$
is a harmonious $(\Z_3,\{3^4,6\},18)$-SDF.
\end{ex}

In what follows, $\F_q$ and $\F_q^*$ will denote the finite field of order $q$ and its multiplicative
group, respectively. 


\begin{defn}
Let $\Sigma$ be a $(G,K,\lambda)$-SDF and let $q\equiv1$ (mod $\lambda$) be a prime power.  
Lift each $B=\{b_1,\dots,b_k\}\in\Sigma$ to a subset $\ell(B)=\{(b_1,c_1),\dots,(b_k,c_k)\}$ 
of $G\times\F_q$ and set ${\cal L}=\{\ell(B) \ | \ B\in\Sigma\}$. 
By definition of a SDF the list of differences of $\cal L$ is of the form 
\begin{equation}\label{Delta}
\Delta{\cal L}=\bigcup_{g\in G}\{g\}\times \Delta_g
\end{equation} 
where each $\Delta_g$ is a list of $\lambda$ elements of $\F_q$.
Thus it makes sense to ask whether there exists a suitable ${q-1\over\lambda}$-subset $S$ of $\F_q$ such that
\begin{equation}\label{GoodLifting}
S\cdot\Delta_g=\F_q^*\quad\forall \ g\in G.
\end{equation} 
If the answer is in the affirmative, we say that $S$ is
a {\it good companion} for $\cal L$ and call $\cal L$ a {\it good lifting} of $\Sigma$.

Assume that $\Sigma$ is harmonious and that $\cal L$ is a good lifting of $\Sigma$.
By definition of a harmonious SDF, both the flatten of $\cal L$ and its projection $\pi({\cal L})$ on $\F_q$ have size $\lambda$.
Thus it may happen that there is a good companion $S$ for $\cal L$ satisfying the identity
\begin{equation}\label{PerfectLifting}
S\cdot\pi({\cal L})=\F_q^*.
\end{equation} 
In this case we say that $S$ is a {\it perfect companion} for $\cal L$ and call $\cal L$ a {\it perfect lifting} of $\Sigma$.
\end{defn}

If, for instance, each $\Delta_g$ and $\pi({\cal L})$ are complete systems 
of representatives for the cosets of the subgroup $S$ of $\F_q^*$ of index $\lambda$, then a 
perfect companion for $\cal L$ is just $S$.

\begin{thm}\label{PL->RDF}
Let $G$ be a group of order $g$, let $q=\lambda n+1$ be a prime power, and let
$\Sigma$ be a $(G,K,\lambda)$-SDF admitting a good lifting $\cal L$ to $G\times\F_q$.
Then there exists a $(G\times\F_q,G\times\{0\},\underline{n}K,1)$-DF which is resolvable
if $\Sigma$ is harmonious and $\cal L$ is perfect.
\end{thm}
\begin{proof}
By definition of a good lifting there is a suitable good companion $S$ for $\cal L$ for which
(\ref{Delta}) and (\ref{GoodLifting}) hold. 
For each block $B=\{(b_1,c_1),\dots,(b_k,c_k)\}$ of ${\cal L}$,  and for each $s\in S$, 
set $B_s=\{(b_1,sc_1),\dots,(b_k,sc_k)\}$. Then set ${\cal F}=\{B_s \ | \ B\in{\cal L}; s\in S\}$. Note
that $\Delta B_s$ is obtainable from $\Delta B$ by multiplying all its elements by $s$. Thus (\ref{Delta}) and (\ref{GoodLifting}) give
$$\Delta{\cal F}=\bigcup_{g\in G}\{g\}\times (S\cdot\Delta_g)=\bigcup_{g\in G}\{g\}\times\F_q^*=(G\times\F_q)\setminus (G\times\{0\})$$
which means that $\cal F$ is a $(G\times\F_q,G\times\{0\},\underline{n}K,1)$-DF.

Now assume that $\Sigma$ is harmonious, $\cal L$ is perfect and $S$ is a perfect companion for $\cal L$. 
In this case the flatten of the difference family $\cal F$ constructed above is 
$S\cdot\pi({\cal L})=\F_q^*$ because (\ref{PerfectLifting}) holds.  
This means that the flatten of ${\cal F}$ is a complete system of representatives for the non-trivial cosets of
$G\times\{0\}$ in $G\times\F_q$, i.e., that $\cal F$ is resolvable.
\end{proof}

The above theorem was essentially applied in \cite{BWW} to construct some infinite series of non-uniform and non-resolvable
relative difference families starting from $(\Z_g,K,\lambda)$-SDFs with $\lambda=2$ or 4 and flatten of size $2g$.
Applying the theorem with the use of the harmonious SDFs of Example \ref{classic} one can obtain 
uniform and resolvable DF's.  This was done in \cite{BZ} implicitly (at that time the notion of a SDF had not yet been 
introduced) and explicitly in \cite{BYW}.

Now note that Theorem \ref{PL->RDF} and Theorem \ref{RDF->PDF} immediately give the following.

\begin{cor}\label{cor}
Assume that there exists a harmonious $(G,K',\lambda)$-SDF admitting a perfect lifting to $G\times\F_q$
with $G$ of order $g$ and $q=\lambda n+1$.
Then there exists a $(G\times\F_q,\underline{gn}K' \ \cup \ \{g\},g)$-PDF which is the base parallel class of a HS$(2,K,gq)$ with
$K$ the underlying set of $K' \ \cup \ \{g\}$.
\end{cor}

\begin{ex}\label{2,2,4}
It is straightforward to check that $\Sigma=\{\{0,0\}, \ \{0,0\}, \ \{0,0,1,1\}\}$ is a $(\Z_2,\{2,2,4\},8)$-SDF.
We try to obtain a perfect lifting of $\Sigma$ to $\Z_2\times\F_{17}$ of the form
$${\cal L}=\bigl{\{}\{(0,a),(0,-a)\}, \ \{(0,b),(0,c)\}, \ \{(0,-b),(0,d),(1,-c),(1,-d)\}\bigl{\}}.$$ The list of differences of $\cal L$ is 
$$\Delta{\cal L}=(\{0\}\times\Delta_0) \ \cup \ (\{1\}\times\Delta_1)$$
with $\Delta_0=\{1,-1\}\cdot\Delta'_0$ and $\Delta_1=\{1,-1\}\cdot \Delta'_1$
where $\Delta'_0=\{2a,b-c,b+d,c-d\}$ and $\Delta'_1=\{2d,b-c,b-d,c+d\}$.
Also, the projection of the flatten of $\cal L$ on $\F_{17}$ is $\pi({\cal L})=\{1,-1\}\cdot\pi'$ with $\pi'=\{a,b,c,d\}$.

Let $C$ be the subgroup of $\F_{17}^*$ of index $4$, that is $C=\{1,4,13,16\}$ and assume that 
\begin{equation}\label{completesystems}
\mbox{$\Delta'_0$, $\Delta'_1$, $\pi'$ are complete systems of
representatives for the cosets of $C$ in $\F_{17}^*$}
\end{equation}
Thus we are assuming that $C\cdot \Delta'_0=C\cdot \Delta'_1=C\cdot\pi'=\F_{17}^*$. 
Observing that $C=S\cdot\{1,-1\}$ with $S=\{1,4\}$,
we see that $S$ is a perfect companion for $\cal L$, i.e., conditions (\ref{GoodLifting}) and (\ref{PerfectLifting}) hold.
Indeed we have
$$S\cdot\Delta_i=S\cdot\{1,-1\}\cdot\Delta'_i=C\cdot\Delta'_i=\F_{17}^*\quad\mbox{for $i=1, 2$};$$
$$S\cdot\pi({\cal L})=S\cdot\{1,-1\}\cdot\pi'=C\cdot\pi'=\F_{17}^*$$

Looking at the cosets of $C$ in $\F_{17}^*$,
$$C=\{1,4,13,16\},\quad2C=\{2, 8, 9,15\},\quad
3C=\{3, 12, 5, 14\},\quad6C=\{6, 7, 10, 11\},$$
it is readily seen that $(2,3,6,4)$ is a quadruple satisfying (\ref{completesystems}). 
Thus
$${\cal L}=\bigl{\{}\{(0,2),(0,15)\}, \ \{(0,3),(0,6)\}, \ \{(0,14),(0,4),(1,11),(1,13)\}\bigl{\}}$$
is a perfect lifting of $\Sigma$ to $\Z_2\times\F_{17}$ with perfect companion $S=\{1,4\}$.
Then, following the instructions of Theorem \ref{PL->RDF}, we get a resolvable $(\Z_2\times\F_{17},\Z_2\times\{0\},\{2^4,4^2\},1)$-DF,
that is ${\cal F}={\cal L} \ \cup \ {\cal L}'$ where ${\cal L}'$ is obtainable from $\cal L$ by multiplying the second coordinates of the elements of all its blocks by 4:
$${\cal L}'=\bigl{\{}\{(0,8),(0,9)\}, \ \{(0,12),(0,7)\}, \ \{(0,5),(0,16),(1,10),(1,1)\}\bigl{\}}$$

\normalsize
Let us identify $\Z_2\times\F_{17}$ with $\Z_{34}$ via the isomorphism (given by the Chinese remainder 
theorem) mapping every $(x,y)$ of $\Z_2\times\F_{17}$ to the element $17x+18y$ of $\Z_{34}$.
Then $\cal F$ can be seen as a resolvable $(\Z_{34},\{0,17\},\{2^4,4^2\},1)$-DF whose blocks are the following:
$$\{2,32\}, \ \{20,6\}, \ \{14,4,11,13\}, \ \{8,26\}, \ \{12,24\}, \ \{22,16,27,1\}$$
Now, by Theorem \ref{RDF->PDF}, we can say that ${\cal P}={\cal F} \ \cup \ \{B+17 \ | \ B\in{\cal F}\} \ \cup \ \{\{0,17\}\}$ is 
a $(\Z_{34},\{2^9,4^4\},2)$-PDF and $\{{\cal P}+i \ | \ 0\leq i\leq 16\}$ is the set of parallel classes of a HS$(2,\{2,4\},34)$
with $153$ blocks of size $2$ and $68$ blocks of size $4$.
\end{ex}

\section{Composition constructions}
It is worth observing that any perfect lifting of a harmonious SDF automatically gives an
infinite series of harmonious linear spaces. Indeed the following result, in the same spirit of Theorem 1 in \cite{W}, holds.

\begin{prop}\label{q->q^k}
If a harmonious SDF in $G$ admits a perfect lifting to $G\times\F_q$, then it also 
admits a perfect lifting to $G\times\F_{q^k}$ for every integer $k$.
\end{prop}
\begin{proof}
Assume that $\Sigma$ is a harmonious SDF in $G$ admitting a perfect lifting 
${\cal L}$ to $G\times\F_q$. Thus, using the same notation as in the previous section, we have
$$\Delta{\cal L}=\bigcup_{g\in G}\{g\}\times \Delta_g,\quad S\cdot\Delta_g=\F_q^* \ \ \forall g\in G,\quad{\rm and}\quad S\cdot \pi({\cal L})=\F_q^*$$
where $S$ is a perfect companion for $\cal L$.

Take a complete system $T$ of representatives for the cosets of $\F_q^*$ in $\F_{q^k}^*$ so that we have $\F_q^*\cdot T=\F_{q^k}^*$.
Then, for every block $L\in\cal L$ and any $t\in T$, let $L_t$ be the subset of $G\times \F_{q^k}^*$
obtainable from $L$ by multiplying the second coordinates of all its elements by $t$. Now set
${\cal L}'=\{L_t \ | \ L\in{\cal L}, t\in T\}$. Clearly, ${\cal L}'$ is a lifting of $\Sigma$ to $G\times\F_{q^k}$. 
Also, we have $\Delta{\cal L}'=\bigcup_{g\in G}\{g\}\times \Delta'_g$ with $\Delta'_g=\Delta_g\cdot T$ for every $g\in G$,
and $\pi({\cal L}')=\pi({\cal L})\cdot T$. Thus we have:
$$S\cdot\Delta'_g=S\cdot \Delta_g\cdot T=\F_q^*\cdot T=\F_{q^k}^* \ \ \forall g\in G\quad{\rm and}\quad 
S\cdot \pi({\cal L}')=S\cdot \pi({\cal L})\cdot T=\F_q^*\cdot T=\F_{q^k}^*$$
which means that $S$ is a perfect companion for ${\cal L}'$, i.e., ${\cal L}'$ is perfect.
\end{proof}

Applying Proposition \ref{q->q^k} together with Example \ref{2,2,4} we get a HS$(2,\{2,4\},2\cdot17^k)$ for every $k$. 

Now we show a composition construction making use of {\it difference matrices} (see, e.g., \cite{BJL} or \cite{C}).
A $(k\times v)$-matrix with entries in an additive group $G$ of order $v$ such that the difference of any two distinct rows
is a permutation of $G$ is said to be a $(G,k,1)$ difference matrix (DM for short). It is called  {\it homogeneous} if each row is also a permutation of $G$.
Adapting the construction \cite{J} for ordinary difference families by the second author we have the following.

\begin{thm}\label{composition}
Let $X$, $Y$ be two groups and let $Z=X\times Y$. If there exist a resolvable $(G\times X,G\times\{0\},K_1,1)$-DF,
a resolvable $(G\times Y,G\times\{0\},K_2,1)$-DF and a homogeneous $(Y,\max(K_1),1)$-DM, then there exists
a resolvable $(G\times Z,G\times\{0\},K,1)$-DF with $K=\underline{|Y|}K_1 \ \cup \ K_2$.
\end{thm}
\begin{proof}
For each block $B=\{(g_1,x_1),\dots,(g_k,x_k)\}$
of a resolvable $(G\times X,G\times\{0\},K_1,1)$-DF, say ${\cal F}_X$, and for each column
$(m_{1c},\dots,m_{kc})^T$ of a homogeneous $(Y,\max(K_1),1)$-DM, say $M=(m_{rc})$,
set $$B\circ M^c=\{(g_1,x_1,m_{1c}),\dots,(g_k,x_k,m_{kc})\}.$$
Also, for each block $B=\{(g_1,y_1),\dots,(g_k,y_k)\}$ of a resolvable $(G\times Y,G\times\{0\},K_2,1)$-DF, 
say ${\cal F}_Y$, set 
$\overline{B}=\{(g_1,0,y_1),\dots,(g_k,0,y_k)\}$.
It is easy to check that
$$\{B\circ M^c \ | \ B\in{\cal F}_X; 1\leq c\leq |Y|\} \ \cup \ \{\overline{B} \ | \ B\in{\cal F}_Y\}$$
is a resolvable $(G\times Z,K,1)$-DF.
\end{proof}

\begin{rem}\label{rem}
In view of Theorem \ref{RDF->PDF},
with the same hypotheses as in Theorem \ref{composition}, we see that 
there also exists a $(G\times Z,\underline{|Z|}K \ \cup \ \{|Z|\},|Z|)$-PDF
and hence a HS$(2,K',v)$ with $K'$
the underlying set of $K \ \cup \ \{|Z|\}$ and $v=|G|\cdot|X|\cdot|Y|$.
\end{rem}

\section{Cyclotomy}
Here we show how to use cyclotomy for getting a perfect lifting of a harmonious SDF. 
Given a prime power $q\equiv1$ (mod $e$), the subgroup of $\F_q^*$ of index $e$ -- that is the group
of non-zero $e$-th powers of $\F_q$ -- will be denoted by $C^e$.
If $r$ is a primitive element of $\F_q$, then the set of cosets of $C^e$ in $\F_q^*$ is $\{r^iC^e \ | \ 0\leq i\leq e-1\}$.
As usual, the coset $r^iC^e$ (called {\it the $i$-th cyclotomic class of order $e$}) will be denoted by $C_i^e$.
It is evident that an $e$-subset $S$ of $\F_q^*$ is a complete system of representatives for the cosets of $C^e$ in $\F_q^*$ if and only
if there is an ordering $(s_0,s_1,\dots,s_{e-1})$ of its elements such that $s_i\in C_i^e$ for each $i$.

From now on, given two positive integers $e$, $n$, we set
$$U=\sum_{h=1}^n{n\choose h}(e-1)^h(h-1)\quad{\rm and}\quad Q(e,n)={1\over4}(U+\sqrt{U^2+4ne^{n-1}})^2.$$ 
We will need the following result deriving from the theorem of Weil on multiplicative character sums (see Theorem 2.2 in \cite{BP}).

\begin{thm}\label{BP}
Let $q\equiv1$ $($mod $e)$ be a prime power, let $\{c_0,c_1,\dots, c_{n-1}\}$ be a $n$-subset of $\F_q$ and
let $(\alpha_0,\alpha_1,\dots,\alpha_{n-1})$ be a $n$-tuple of elements of $\Z_e$. 
Then the set
$$X=\{x\in \F_q \ : \ x-c_j\in C_{\alpha_j}^e \ \mbox{for }  0\leq j\leq n-1\}$$ is non-empty if $q>Q(e,n)$.
\end{thm}

In other words, whenever a prime power $q\equiv1$ (mod $e$)
is greater than $Q(e,n)$, the existence of an element $x\in \F_q$ satisfying $n$ cyclotomic conditions of the
form $x-c_j\in C^e_{\alpha_j}$ is guaranteed provided that the $c_j$'s are pairwise distinct.

The following theorem is a variation of Theorem 7.2 in \cite{BN}.

\begin{lem}\label{weil}
Every harmonious $(G,K,\lambda)$-SDF admits a perfect lifting to $G\times\F_q$ 
for all prime powers $q\equiv\lambda+1$ $($mod $2\lambda)$ greater than $Q(\lambda,\max(K))$.
\end{lem}
\begin{proof}
Let $\Sigma=\{B_1,\dots,B_\sigma\}$ be a harmonious $(G,K,\lambda)$-SDF with $|B_h|=k_h$ and 
$B_h=\{b_{h,1},\dots,b_{h,k_h}\}$ for $1\leq h\leq \sigma$.

Let $S$ be the set of all pairs $(h,i)$ with $h\in\{1,\dots,\sigma\}$ and $i\in\{1,\dots,k_h\}$.
This set has size $k_1+k_2+\dots+k_\sigma=\lambda$ since $\Sigma$ is harmonious, hence we may choose a
bijection $\phi: S  \longrightarrow \Z_\lambda$.

Let $G_0=\{g\in G\ | \ 2g=0\}$ be the set of all involutions of $G$ together with $0$.
Obviously, $G\setminus G_0$ can be partitioned into pairs $\{g,-g\}$ of opposite and distinct elements of $G$. 
Thus we can write $G=G_0 \ \cup \ G^+ \ \cup \ G^-$ where $G^+$ is a complete set of representatives for the
above pairs and $G^-=-G^+$.

Let $T$ be the set of all triples $(h,i,j)$ with $h\in\{1,\dots,\sigma\}$ and $i$, $j$
distinct elements of $\{1,\dots,k_h\}$. Then consider the bijection
$$^{^{\overline{ \ \ \ }}}: (h,i,j)\in T \longrightarrow \overline{(h,i,j)} = (h,j,i)\in T$$
which is clearly involutory and without fixed elements.

For every $g\in G$, let $T_g$ be the set of triples $(h,i,j)$ of $T$ such that $b_{h,i}-b_{h,j}=g$.
By definition of a $(G,K,\lambda)$-SDF, $\{T_g \ | \ g\in G\}$ is a partition of $T$ into subsets of size $\lambda$.
We have:
$$(h,i,j)\in T_g \Longleftrightarrow b_{h,i}-b_{h,j}=g \Longleftrightarrow b_{h,j}-b_{h,i}=-g \Longleftrightarrow (h,j,i)\in T_{-g}$$
hence
\begin{equation}\label{conjugate}
\overline{T_g}=T_{-g} \quad \forall g\in G
\end{equation}
For any $g\in G^+$, let us choose an arbitrary bijection $\psi_g: T_g \longrightarrow \Z_\lambda$.
In view of (\ref{conjugate}) and recalling that the index of a SDF is always even, it makes sense to consider the map 
\begin{equation}\label{conjugate2}
\psi_{-g}: (h,i,j)\in T_{-g} \longrightarrow \psi_g(h,j,i)+{\lambda\over2}\in \Z_\lambda
\end{equation}
The fact that $\psi_g$ is a bijection implies that $\psi_{-g}$ is a bijection as well.

Given $g\in G_0$, we have $g=-g$ and hence $\overline{T}_g=T_{g}$ in view of (\ref{conjugate}).
Thus we  can write $T_g=T'_g \ \cup \ \overline{T'}_g$ for a suitable ${\lambda\over2}$-subset
$T'_g$ of $T_g$. Let $\Z'_\lambda$ be a system of representatives for the
cosets of $\{0,{\lambda\over2}\}$ in $\Z_\lambda$ and let us choose any bijection
$\psi'_g: T'_g \longrightarrow \Z'_\lambda$. Then consider the map 
\begin{equation}\label{conjugate3}
\overline{\psi'}_g:  (h,i,j)\in\overline{T'}_g \longrightarrow \psi'_g(h,j,i)+{\lambda\over2}\in\Z_\lambda\setminus\Z'_\lambda
\end{equation}
Also here, the fact that $\psi'_g$ is a bijection assures that $\overline{\psi'}_g$ is a bijection too. 
The map $\psi_g: T_g \longrightarrow \Z_\lambda$ 
defined by 
$$\psi_g(h,i,j)=\begin{cases}\psi'_g(h,i,j) & {\rm if} \ (h,i,j)\in T'_g \medskip\cr \overline{\psi'}_g(h,i,j) & {\rm if} \ (h,i,j)\in \overline{T'}_g\end{cases}$$
is clearly a bijection as well. 

Finally, let $\psi: T \longrightarrow \Z_\lambda$ be the map whose restriction to $T_g$ coincides with $\psi_g$ for every $g\in G$.
In view of (\ref{conjugate2}) and (\ref{conjugate3}) this map has the following property:
\begin{equation}\label{psi}
\psi(h,j,i)=\psi(h,i,j)+{\lambda\over2} \quad \forall (h,j,i)\in T
\end{equation}

Now let $q$ be a prime power as in the statement and let us lift each $B_h$ to a subset 
$\ell(B_h)=\{(b_{h,1},c_{h,1}),\dots,(b_{h,k_h},c_{h,k_h})\}$ of $G\times \F_q$
by taking the first element $c_{h,1}$ arbitrarily in $\F_q^*$ and then by taking the other elements 
$c_{h,2}$, $c_{h,3}$, \dots, $c_{h,k_h}$ iteratively, one by one, according to the rule that
once that $c_{h,i-1}$ has been chosen, we pick $c_{h,i}$ arbitrarily in the set 
$$X_{h,i}=\{x\in C^{\lambda}_{\phi(h,i)} \ : \ x-c_{h,j}\in C^\lambda_{\psi(h,i,j)} \quad {\rm for} \ 1\leq j\leq i-1\}$$
where $\phi$ is the bijection from $S$ to $\Z_\lambda$ chosen at the beginning of this proof.
It is convenient to rewrite $X_{h,i}$ as
$$X_{h,i}=\{x\in \F_q^* \ : x-0\in C^{\lambda}_{\phi(h,i)} \ {\rm and} \ x-c_{h,j}\in C^\lambda_{\psi(h,i,j)} \quad {\rm for} \ 1\leq j\leq i-1\}.$$
The above choice of $c_{h,i}$ can be actually done since $X_{h,i}$ is not empty by Lemma \ref{BP}. 
Indeed, as shown below, $\{0,c_{h,1},...,c_{h,i-1}\}$ is a set and $q>Q(\lambda,i)$:
\begin{itemize}
\item if $\{0,c_{h,1},...,c_{h,i-1}\}$ had repeated elements we would have
either $0=c_{h,j}$ with $1\leq j\leq i-1$ contradicting that
$c_{h,j}$ had been picked in $X_{h,j}\subset C^{\lambda}_{\phi(h,j)}\subset\F_q^*$, or
$0=c_{h,j_1}-c_{h,j_2}$ with $1\leq j_1<j_2\leq i-1$ contradicting that 
$c_{h,j_2}-c_{h,j_1} \in C^\lambda_{\psi(h,j_2,j_1)}\subset\F_q^*$ since $c_{h,j_2}$ had been picked in $X_{h,j_2}$;
\item we have $\max(K) \geq k_h \geq i$ and then $q>Q(\lambda,\max(K))>Q(\lambda,i)$ since 
the function $Q(\lambda,n)$ defined before Theorem \ref{BP} is clearly increasing in the variable $n$.
\end{itemize}

Now we show that the lifting ${\cal L}=\{\ell(B_1),\dots,\ell(B_\sigma)\}$ of $\Sigma$ to $G\times \F_q$ is perfect.
First observe that we have 
\begin{equation}\label{c-c}
c_{h_,i}-c_{h,j}\in C^\lambda_{\psi(h,i,j)}\quad \forall (h,i,j)\in T
\end{equation}
This is clear if $i>j$ considering the rule used for selecting the $c_{h,i}$'s.
If $i<j$, for the same reason, we have $c_{h_,j}-c_{h,i}\in C^\lambda_{\psi(h,j,i)}$ and then, by (\ref{psi}), we can write
\begin{equation}\label{c-c2}
c_{h_,j}-c_{h,i}\in C^\lambda_{\psi(h,i,j)+\lambda/2}
\end{equation}
The hypothesis $q\equiv\lambda+1$ (mod $2\lambda$) implies that 
$-1\in C^\lambda_{\lambda/2}$. Thus, multiplying (\ref{c-c2}) by $-1$ we get (\ref{c-c}) again.

Now we have $$\bigcup_{h=1}^\sigma\Delta\ell(B_h)=\bigcup_{g\in G}\{g\}\times\Delta_g$$
with $\Delta_g=\{c_{h,i}-c_{h,j} \ | \ (h,i,j)\in T_g\}$. Thus, considering (\ref{c-c}) and 
the fact that $\psi$ is bijective on each $T_g$, we see that  $\Delta_g$ is a 
complete system of representatives for the cosets of $C^\lambda$ in $\F_q^*$ 
whichever is $g\in G$. Hence we have $C^\lambda\cdot\Delta_g=\F_q^*$ for every $g\in G$, i.e., 
$C^\lambda$ is a good companion for $\cal L$.
Finally, the projection of the flatten of $\cal L$ on $\F_q$ is
$\pi({\cal L})=\{c_{h,i} \ | \ (h,i)\in S\}$ and then, recalling that $c_{h,i}\in C^{\lambda}_{\phi(h,i)}$
for every $(h,i)\in S$ and that $\phi: S \longrightarrow \Z_\lambda$ is bijective, we can
say that $\pi({\cal L})$ is also a complete system of representatives for the cosets of $C^\lambda$ in $\F_q^*$, i.e., $C^\lambda\cdot\pi({\cal L})=\F_q^*$.
We conclude that $C^\lambda$ is a perfect companion for $\cal L$. 
\end{proof}

Putting together Lemma \ref{weil} and Corollary \ref{cor} we obtain the following.

\begin{thm}\label{HSDF->HS}
If there exists a harmonious $(G,K',\lambda)$-SDF, then there exists a HS$(2,K,gq)$, with $g=|G|$ and $K$ the underlying set
of $K' \ \cup \ \{g\}$, for every prime power $q\equiv\lambda+1$ $($mod $2\lambda)$ greater than $Q(\lambda,\max(K'))$.
\end{thm}


\section{Some explicit results}

Unfortunately, the harmonious linear spaces realizable via Theorem \ref{HSDF->HS} are not ``concrete" in the sense that they have a huge number of points.
Indeed the values of $Q(e,n)$ are very large in general. Yet, as shown in this section, it is sometimes possible to get concrete constructions
if the block sizes are small.

\begin{thm}\label{2,2,4bis}
For any prime power $q=8n+1>9$ there exists a $(\Z_2\times\F_q,\{2^{4n+1},4^{2n}\},2)$-PDF which is the starter parallel class of a HS$(2,\{2,4\},2q)$.
\end{thm}
\begin{proof}
Let $\Sigma=\{\{0,0\}, \ \{0,0\}, \ \{0,0,1,1\}\}$ be the $(\Z_2,\{2,2,4\},8)$-SDF
considered in Example \ref{2,2,4} and consider a lifting of $\Sigma$ to $\Z_2\times\F_q$ of the following form
$${\cal L}=\bigl{\{}\{(0,a),(0,b)\}, \ \{(0,c),(0,d)\}, \ \{(0,-a),(0,-c),(1,-b),(1,-d)\}\bigl{\}}.$$ 
The list of differences of $\cal L$ is 
$$\Delta{\cal L}=(\{0\}\times\Delta_0) \ \cup \ (\{1\}\times\Delta_1)$$
with $\Delta_0=\{1,-1\}\cdot\Delta'_0$ and $\Delta_1=\{1,-1\}\cdot \Delta'_1$
where $\Delta'_0=\{a-b,c-d,a-c,b-d\}$ and $\Delta'_1=\{a-b,a-d,b-c,c-d\}$.
The projection of the flatten of $\cal L$ on $\F_q$ is $\pi({\cal L})=\{1,-1\}\cdot\pi'$ with $\pi'=\{a,b,c,d\}$.

Reasoning exactly as in Example \ref{2,2,4} we can see that  if 
$(a,b,c,d)$ is a quadruple of elements of $\F_q^*$ such that  
$\Delta'_0$, $\Delta'_1$, $\pi'$ are complete systems of representatives for the cosets of $C^4$ in $\F_q^*$,
then $\cal L$ is a perfect lifting with perfect companion a complete system of representatives for the cosets of 
$\{1,-1\}$ in $C^4$. A quadruple $(a,b,c,d)$ as above will be called ``good".
If $q=17$, it is possible to see that no good quadruple exists. In this case, however, the assertion is true 
because of Example \ref{2,2,4} where $\Sigma$ was lifted in a different way.
For $q>Q(4,4)$ the existence of a good quadruple can be proved as follows.

Take $a$ in $C^4_0$. Then take $b$ in the set $\{x\in \F_q^* \ : \ x\in C^4_1, x-a\in C^4_0\}$ which is not empty
by Theorem \ref{BP} since we have $q>Q(4,4)>Q(4,2)$. Now take $c$ in the set $\{x\in \F_q^* \ : \ x\in C^4_2, x-a\in C^4_1, x-b\in C^4_1\}$
which is not empty by Lemma \ref{BP} since we have $q>Q(4,4)>Q(4,3)$. Finally, take $d$ in the set 
$\{x\in \F_q^* \ : \ x\in C^4_3, x-a\in C^4_3, x-b\in C^4_3, x-c\in C^4_2\}$ which, again, is not empty by Lemma \ref{BP} 
since we have $q>Q(4,4)$. It is quite clear that the quadruple $(a,b,c,d)$ meets the desired requirements.

One can check that $Q(4,4)=263681$ which, fortunately, is not a huge number;
there are ``only" 5852 prime powers $q\equiv1$ (mod 8) in the range $]17,26381]$.
We have found by computer a good quadruple $(a,b,c,d)$ of $\F_q$ for each of these prime powers $q$.
Here is, for instance, a good quadruple of $\F_q$ for $q<100$.
\begin{center}
\begin{tabular}{|l|c|r|c|c|}
\hline {$q$} & $(a,b,c,d)$   \\
\hline 
\hline $25$ & $(1,g,g+1,4g+1)$ \ \mbox{$g$ root of the primitive polynomial $x^2+x+2$} \\
\hline $41$ & $(1,12,17,32)$ \\
\hline $49$ & $(1,g,3g+3,5g+4)$ \ \mbox{$g$ root of the primitive polynomial $x^2+x+3$} \\
\hline $73$ & $(1,3,20,62)$ \\
\hline $81$ & $(g^3+1,g^3+g+1,2g^2+2,g^3+g+2)$ \\
& \mbox{$g$ root of the primitive polynomial $2x^4+x^3+1$} \\
\hline $89$ & $(1,3,15,55)$ \\
\hline $97$ & $(1,2,20,46)$ \\
\hline
\end{tabular} 
\end{center}

\end{proof}

\begin{thm}\label{3,18bis}
If $q=18n+1$ is a prime power with $n$ odd, then there exists a $(\Z_3\times\F_q,\{3^{12n+1},6^{3n}\},3)$-PDF which is the starter parallel class of a HS$(2,\{3,6\},3q)$.
\end{thm}
\begin{proof}
Let $\Sigma=\{B_1,\dots,B_5\}$ be the harmonious $(\Z_3,\{3^4,6\},18)$-SDF of Example \ref{3,18}, let
$q=18n+1$ be a prime power with $n$ odd, and let $\varepsilon$ be a primitive cube root of unity in $\F_q$.
Consider a lifting ${\cal L}=\{L_1,\dots,L_5\}$ of $\Sigma$ to $\Z_3\times\F_q$ of the following form:
$$L_1=\{(0,a),(0,a\varepsilon),(0,a\varepsilon^2)\},\quad\quad L_2=\{(0,b),(0,b\varepsilon),(0,b\varepsilon^2)\},$$
$$L_3=\{(0,c),(1,c\varepsilon),(2,c\varepsilon^2)\},\quad\quad L_4=\{(0,d),(1,d\varepsilon),(2,d\varepsilon^2)\},$$
$$L_5=\{(0,1),(0,e),(1,\varepsilon),(1,e\varepsilon),(2,\varepsilon^2),(2,e\varepsilon^2)\}.$$
The list of differences of $\cal L$ is $$\Delta{\cal L}=\bigcup_{i=0}^2\{i\}\times \Delta_i$$ 
with $\Delta_i=\langle\varepsilon\rangle\cdot \Delta'_i$ for $i=0,1,2$ where
$$\Delta'_0=\{\pm a(\varepsilon-1),\pm b(\varepsilon-1),\pm(e-1)\};$$ 
$$\Delta'_1=-\Delta'_2=\{c(\varepsilon-1),d(\varepsilon-1),
\varepsilon-1,e(\varepsilon-1),\varepsilon-e,e\varepsilon-1\}.$$
Also, it is readily seen that the projection of the flatten of $\cal F$ on $\F_q$ is
$$\pi({\cal L})=\langle\varepsilon\rangle\cdot \pi' \quad{\rm with} \quad \pi'=\{1,a,b,c,d,e\}.$$ 


Assume that each of the three sixtuples $\Delta'_0$, $\Delta'_1$ and $\pi'$ is a complete system of representatives for the 
cosets of $C^6$ in $\F_q^*$. In this case we have 
\begin{equation}\label{C^6Delta}
C^6\cdot\Delta'_i=\F_q^*\quad{\rm for} \ i=0,1,2\quad{\rm and}\quad C^6\cdot \pi'=\F_q^*
\end{equation}
Let $S$ be a complete system of representatives for the cosets of $\langle\varepsilon\rangle$
in $C^6$, so that we have 
\begin{equation}\label{S<epsilon>}
S\cdot\langle\varepsilon\rangle=C^6
\end{equation} 
Using (\ref{C^6Delta}) and (\ref{S<epsilon>}) we get:
$$S\cdot\Delta_i=S\cdot\langle\varepsilon\rangle\cdot\Delta'_i=C^6\cdot\Delta'_i=\F_q^*$$
$${\rm and}$$
$$S\cdot\pi({\cal L})=S\cdot\langle\varepsilon\rangle\cdot \pi'=C^6\cdot \pi'=\F_q^*.$$
This means that $S$ is a perfect companion for $\cal L$, i.e., $\cal L$ is a perfect lifting of $\Sigma$.
Hence, by Theorem \ref{PL->RDF} there exists a resolvable $(\Z_3\times\F_q,\Z_3\times\{0\},\{3^{4n},6^{n}\},1)$-DF and 
then, by Theorem \ref{RDF->PDF}, there exists a
$(\Z_3\times\F_q,\{3^{12n+1},6^{3n}\},3)$-PDF which is the base parallel class of a HS$(2,\{3,6\},3q)$.

Thus, for proving the assertion it is enough to show that 
there is at least one ``good" quintuple $(a,b,c,d,e)$ satisfying conditions (\ref{C^6Delta}).
We found it explicitly for $q<Q(6,4)=9152353$ by computer search.
For $q>Q(6,4)$ it is enough to reason as follows.
Take $(a,b,c,d)$ arbitrarily in the cartesian product $C^6_1\times C^6_2\times C^6_3\times C^6_4$.
Then, if $C^6_i$ is the cyclotomic class of order 6 containing $\varepsilon-1$,
take $e$ in the set $$\{x\in \F_q^* \ : \ x\in C^6_{5}, x-1\in C^6_i, x-\varepsilon\in C^6_{i+4}, x-\varepsilon^2\in C^6_{i+2}\}$$
which, by Lemma \ref{BP}, is not empty since $q>Q(6,4)$. Taking into account that $-1\in C^6_3$ since $n$ is odd and that
$\varepsilon$ is obviously in $C^6_0$, it is easy to see that $\Delta'_0$, $\Delta'_1$ and $\pi'$ are evenly distributed over the cosets $C^6_0$, \dots, $C^6_5$,
i.e., $(a,b,c,d,e)$ is good.
\end{proof}

\begin{ex}
Let us see how to apply the above theorem in the smallest case $q=19$.
Taking $(a,b,c,d,e)=(2,4,5,8,10)$ as good quintuple satisfying (\ref{C^6Delta}),
and taking $\varepsilon=7$ as primitive cube root of unity of $\F_{19}$, the blocks of 
the perfect lifting $\cal L$ are the following: 
$$L_1=\{(0,2),(0,14),(0,3)\},\quad\quad L_2=\{(0,4),(0,9),(0,6)\},$$
$$L_3=\{(0,5),(1,16),(2,17)\},\quad\quad L_4=\{(0,8),(1,18),(2,12)\},$$
$$L_5=\{(0,1),(0,10),(1,7),(1,13),(2,11),(2,15)\}.$$

In this case $\langle\varepsilon\rangle$ coincides with $C^6$, hence we can take $S=\{1\}$. 
Thus, by Theorem \ref{PL->RDF}, $\cal L$ itself is a resolvable $(\Z_3\times\F_{19},\Z_3\times\{0\},\{3^4,6\},1)$-DF. 
Now, let us identify $\Z_3\times\F_{19}$ with $\Z_{57}$ via the isomorphism (given by the Chinese remainder 
theorem) mapping every $(x,y)$ of $\Z_3\times\F_{19}$ to the element $19x-18y$ of $\Z_{57}$.
In this way $\cal L$ can be viewed as a resolvable $(\Z_{57},19\Z_{57},\{3^{4},6\},1)$-DF  whose blocks are
$$\{21, 33, 3\}, \{42, 9, 6\}, \{24, 16, 17\}, \{27, 37, 50\}, \{39, 48, 7, 13, 11, 53\}.$$
At this point, applying Theorem \ref{RDF->PDF}, we get a $(\Z_{57},\{3^{13},6^3\},3)$-PDF 
whose blocks are $\{0,19,38\}$ and the following

$$\{21, 33, 3\}, \{42, 9, 6\}, \{24, 16, 17\}, \{27, 37, 50\}, \{39, 48, 7, 13, 11, 53\},$$
$$\{40, 52, 22\}, \{4, 28, 25\}, \{43, 35, 36\}, \{46, 56, 12\}, \{1, 10, 26, 32, 30, 15\},$$
$$\{2, 14, 41\}, \{23, 47, 44\}, \{5, 54, 55\}, \{8, 18, 31\}, \{20, 29, 45, 51, 49, 34\}.$$

The above PDF is the starter parallel class of a HS$(2,\{3,6\},57)$ with 247 blocks of size 3 
and 57 blocks of size 6.
\end{ex}

Finally, we give a strong indication about the existence of a HS$(2,\{3,4,5,6\},56n+4)$ with $14n+1$ blocks of size 4 and
$56n^2+4n$ blocks of size $k$ for each $k\in\{3,5,6\}$ whenever $q=14n+1$ is a 
prime power greater than 43. Let $\Sigma$ be the harmonious $(\Z_2^2,\{3,5,6\},14)$-SDF of Example \ref{4,14}.
Hence, rewriting the elementary abelian group in the natural way using elements $(0,0)$, $(0,1)$, $(1,0)$ and $(1,1)$,
the blocks of $\Sigma$ are: 

$\{(0,0),(0,0),(0,0)\}$;

 $\{(0,0),(0,0),(0,1),(1,0),(1,1)\}$;
 
 $\{(0,1),(0,1),(1,0),(1,0),(1,1),(1,1)\}\}.$

Consider a lifting $\cal L$ of $\Sigma$ to $\Z_2^2\times \F_q$ of the following form:

$\{(0,0,a), (0,0,b), (0,0,c)\}$

$\{(0,0,d), (0,0,e), (0,1,g), (1,0,-g), (1,1,f)\}$

$\{(0,1,-a), (0,1,-d), (1,0,-b), (1,0,-e), (1,1,-c), (1,1,-f)\}$

It is straightforward to check that $$\Delta{\cal L}=\bigcup_{g\in\Z_2^2}\{g\}\times (\{1,-1\}\cdot\Delta'_g)\quad{\rm and}\quad \pi({\cal L})=\{1,-1\}\cdot\pi'({\cal L})$$
where
$$\Delta'_{(0,0)}=\{a-b,a-c,b-c,d-e,a-d,b-e,c-f\};$$
$$\Delta'_{(0,1)}=\{d-g,e-g,f+g,b-c,b-f,c-e,e-f\};$$
$$\Delta'_{(1,0)}=\{d+g,e+g,f-g,a-c,a-f,c-d,d-f\};$$
$$\Delta'_{(1,1)}=\{d-f,e-f,2g,a-b,a-e,b-d,d-e\};$$
$$\pi'({\cal L})=(a,b,c,d,e,f,g).$$
Reasoning as in Theorem \ref{2,2,4bis} and Theorem \ref{3,18bis}, the desired HS is obtainable provided that
the field elements $a,b,c,d,e,f,g$ are taken in such a way that each of the above 7-tuples is evenly distributed 
over the cosets of $C^7$. Still reasoning as in the above theorems, the existence of a ``good" $(a,b,c,d,e,f,g)$ 
is guaranteed for $q>Q(7,6)$. This number, however, is so large (close to $10^{13}$) that a computer 
search for all prime powers $q\equiv1$ (mod 14) smaller than $Q(7,6)$ is probably too hard (and maybe useless).
On the other hand, a computer search done for the small values of $q$ until $10^4$ 
provides strong evidence for conjecturing that a good $(a,b,c,d,e,f,g)$ always exists provided that $q\geq71$. 
Indeed, while we could not find any good 7-tuple in the smallest cases $q=29$ and $q=43$, the number of good 
7-tuples rapidly increases for $71\leq q\leq 9941$.
In the following table we report an example of a good $(a,b,c,d,e,f,g)$ for $q<600$.
\begin{center}
\begin{tabular}{|l|c|r|c|c|}
\hline {$q$} & $(a,b,c,d,e,f,g)$   \\
\hline 
\hline $71$ & $(1,15,44,50,63,53,9)$ \\
\hline $113$ & $(1,2,90,8,12,104,63)$ \\
\hline $127$ & $(1,8,54,112,67,5,48)$ \\
\hline $197$ & $(1,2,124,24,32,179,35)$ \\
\hline $211$ & $(1,2,13,95,120,53,126)$\\
\hline $239$ & $(1,2,70,155,185,235,129)$ \\
\hline $281$ & $(1,2,12,206,265,247,55)$ \\
\hline
\end{tabular} 
\quad\quad\quad
\begin{tabular}{|l|c|r|c|c|}
\hline {$q$} & $(a,b,c,d,e,f,g)$   \\
\hline 
\hline $337$ & $(1,2,10,255,34,113,228)$ \\
\hline $379$ & $(1,2,16,279,154,146,112)$ \\
\hline $421$ & $(1,2,7,41,55,57,105)$ \\
\hline $449$ & $(1,2,4,258,270,149,413)$ \\
\hline $463$ & $(1,2,18,405,263,307,175)$\\
\hline $491$ & $(1,2,4,118,439,407,454)$ \\
\hline  $547$ & $(1,2,4,274,101,247,498)$\\
\hline
\end{tabular} 
\end{center}

Thus we can state the following.
\begin{thm}
If $q=14n+1$ is a prime power with $71\leq q\leq 9941$ or $q>Q(7,6)$, then there exists a $(\Z_2\times\Z_2\times\F_q,\{4^1,3^{4n},5^{4n},6^{4n}\},4)$-PDF 
which is the starter parallel class of a HS$(2,\{3,4,5,6\},3q)$.
\end{thm}

\section{Main result}

It is quite evident that the union of some harmonious SDFs in the same group $G$ is still 
a harmonious SDF whose index is the sum of their indices. 

\begin{lem}\label{sum}
If there exist harmonious $(G,K_i,\lambda_i)$-SDFs for $1\leq i\leq n$, then their union is
a harmonious $(G,K,\lambda)$-SDF with $K=\bigcup_{i=1}^nK_i$ and $\lambda=\sum_{i=1}^n\lambda_i$.
\end{lem}

\begin{lem}\label{fund}
For any group $G$ of order $k$ and any integer $h>k+1$, there exists a harmonious 
\begin{center}$(G,\{k^{h(h-k-1)},h^k\},hk(h-k))$-SDF.\end{center}
\end{lem}
\begin{proof}
Consider the multisets $X$, $Y$, $Z$ on $G$ defined as follows 
$$X=O_k;\quad Y=G;\quad Z=O_{h-k} \ \cup \ G.$$
Trivially, we have $\Delta X=O_{k(k-1)}$ and $\Delta Y=\underline{k}(G\setminus\{0\})$ since the 
set of all elements of any group $G$ is a difference set in $G$ of index $|G|$. 
Also, rewriting $Z$ as $O_{h-k+1} \ \cup \ (G\setminus\{0\})$ we
see that $$\Delta Z=\Delta O_{h-k+1} \ \cup \ \Delta(G\setminus\{0\}) \ \cup \ [(G\setminus\{0\})-O_{h-k+1}] \ \cup [O_{h-k+1} - (G\setminus\{0\})].$$
We have
$\Delta O_{h-k+1}=O_{(h-k+1)(h-k)}$ and $\Delta(G\setminus\{0\})=\underline{k-2}(G\setminus\{0\})$
since the set of all non-zero elements of any group $G$ is a difference set in $G$ of index $|G|-2$. 
Finally $(G\setminus\{0\})-O_{h-k+1}=O_{h-k+1} - (G\setminus\{0\})=\underline{h-k+1}(G\setminus\{0\})$. Hence we can write
$$\Delta Z=O_{(h-k+1)(h-k)} \ \cup \ \underline{2h-k}(G\setminus\{0\}).$$
Now consider the family $\Sigma=\{X^{h-k}, Y^{h^2-hk-2h+k}, Z^k\}$. The above informations about $\Delta X$,  $\Delta Y$ and  $\Delta Z$
allow us to say that $\Delta \Sigma=O_{\mu_0} \ \cup \ \mu_1(G\setminus\{0\})$ with
$$\mu_0=(h-k)k(k-1)+k(h-k+1)(h-k)=hk(h-k);$$
$$\mu_1=(h^2-hk-2h+k)k+k(2h-k)=hk(h-k).$$
This means that $\Sigma$ is a SDF of index $\lambda=hk(h-k)$. Its blocks have sizes $k$ and $h$. Those of size $k$ are the 
$h-k$ copies of $X$ and the $h^2-hk-2h+k$ copies of $Y$ for a total of $(h-k)+(h^2-hk-2h+k)=h(h-k-1)$ blocks of size $k$; 
those of size $h$ are the $k$ copies of $Z$.
Thus the flatten of $\Sigma$ has size equal to 
$$h(h-k-1)k+hk=hk(h-k)=\lambda.$$
Thus $\Sigma$ is harmonious and the assertion follows.
\end{proof}
We are now ready to proof our main theorem.
\begin{thm}\label{main}
For any finite non-singleton subset $K$ of $\Z^+$, there are infinitely many values of $v$
for which there exists a harmonious $S(2,K,v)$.
\end{thm}
\begin{proof}
Write $K=\{k_0,k_1,\dots,k_t\}$ with $k_0<k_1<\dots<k_t$, and let $G$ be any group of order $k_0$. 
Distinguish two cases according to whether $k_1>k_0+1$
or $k_1=k_0+1$.

1st case: $k_1>k_0+1$.

For $i=1,...,t$, using Lemma \ref{fund} we can construct a harmonious $(G,\{k_0^{\alpha_i},k_i^{\beta_i}\},\lambda_i)$-SDF,
say $\Sigma_i$, for suitable positive integers $\alpha_i$, $\beta_i$, $\lambda_i$.
By Lemma \ref{sum}, the union $\Sigma$ of the $\Sigma_i$s is a harmonious
$(G,K',\lambda)$-SDF with $K'=\{k_0^\alpha,k_1^{\beta_1},...,k_t^{\beta_t}\}$ where $\alpha=\sum_{i=1}^t\alpha_i$
and $\lambda=\sum_{i=1}^t\lambda_i$. Thus, by Theorem \ref{HSDF->HS}, there exists a HS$(2,K,v)$ with
$v=k_0q$ and $q$ any prime power congruent to $\lambda+1$ (mod $2\lambda$) and greater than $Q(\lambda,k_t)$. The assertion then follows considering that there are infinitely 
many primes $q\equiv\lambda+1$ (mod $2\lambda$) by Dirichlet's theorem on arithmetic progressions.

2nd case: $k_1=k_0+1$

It is enough to proceed as in the 1st case by replacing $\Sigma_1$ with a harmonious $(G,k_1,\lambda_1)$-SDF
(which exists in view of Example \ref{classic}).
\end{proof}

\section{Conclusion}
Here, as in \cite{BBGRT}, the maximal prime power factors of an integer $v$ will be called {\it components} of $v$, and
$\F_v$ will denote the ring which is the direct product of all the fields  whose orders are the components of $v$. 
Note that every harmonious linear space constructed in the above sections 
comes from a resolvable $(G\times\F_q,G\times\{0\},K,1)$-DF for suitable $G$, $q$ and $K$.

Thus, considering that the multiplication table of $\F_q$ deprived of the zero-row is a homogeneous $(\F_q,q-1,1)$-DM 
for any prime power $q$, the iterated use of Remark \ref{rem} allows us to get the following results.

\begin{thm}
There exists a S$(2,\{2,4\},2v)$ which is harmonious under $\Z_2\times\F_{v}$ provided that $q>9$ 
and $q\equiv1$ $($mod $8)$ for every component $q$ of $v$.
\end{thm}
\begin{thm}
There exists a S$(2,\{3,6\},3v)$ which is harmonious under $\Z_3\times\F_{v}$ provided that
$q\equiv19$ $($mod $36)$ for every component $q$ of $v$.
\end{thm}
\begin{thm}
There exists a S$(2,\{3,4,5,6\},4v)$ which is harmonious under $\Z_2\times\Z_2\times\F_{v}$ provided that
$q\equiv1$ $($mod $14)$ and $q\in[71,9941] \ \cup \ ]Q(7,6),\infty)$ for every component $q$ of $v$.
\end{thm}
\begin{thm}
Let $K$ be a non-singleton set of positive integers and let $G$ be any group of order $g:=\min(K)$.
Then there is a suitable integer $\lambda$ for which there exists a S$(2,K,gv)$ which is harmonious under $G\times \F_v$
provided that we have $q\equiv \lambda+1$ $($mod $2\lambda)$ and $q>Q(\lambda,\max(K))$ for every component $q$ of $v$.
\end{thm}

We conclude observing that the set of subsets $K$ of $\Z$ for which there exists a PDF whose block sizes 
are precisely the elements of $K$ was never determined before. 
Here, as a consequence of our main result Theorem \ref{main}, we can state the following.

\begin{thm}
The sets $K$ for which there exists a partitioned difference family whose set of block sizes is $K$
are all finite subsets of $\Z^+$ with at least two elements.
\end{thm}

\section*{Acknowledgement}
This work has been performed under the auspices of the G.N.S.A.G.A. of the C.N.R. (National Research Council) of Italy.

\end{document}